\documentclass[12pt,english,oneside,reqno]{amsart}
\usepackage{graphicx} 
\usepackage{amsmath, latexsym, amssymb, url, hyperref}
\usepackage{amsfonts,mathrsfs}
\usepackage{comment}
\usepackage{mathtools}
\newcommand{\nc}{\newcommand}
\newcommand{\rnc}{\renewcommand}
\nc{\R}{\mathbb R}
\nc{\C}{\mathbb C}
\nc{\N}{\mathbb N}
\nc{\Z}{\mathbb Z}
\nc{\Q}{\mathbb Q}
\rnc{\P}{\mathbb P}
\nc{\F}{\mathbb F}
\nc{\Frac}{\mathrm{Frac}}
\rnc{\O}{\mathcal O}
\nc{\Tr}{\mathrm{Tr}}
\nc{\rmU}{\mathrm{U}}
\nc{\rank}{\mathrm{rank}}
\newcommand\cA{{\mathcal A}}

\newcommand\cD{{\mathcal D}}
\newcommand\cE{{\mathcal E}}
\newcommand\cF{{\mathcal F}}
\newcommand\cG{{\mathcal G}}
\newcommand\cH{{\mathcal H}}

\newcommand\cJ{{\mathcal J}}

\newcommand\cM{{\mathcal M}}
\newcommand\cN{{\mathcal N}}

\newcommand\cV{{\mathcal V}}

\newcommand\cX{{\mathcal X}}

\newcommand\cZ{{\mathcal Z}}
\newcommand\cGamma{{\mathit \Gamma}}

\newcommand\bE{{\mathbb E}}

\newcommand\bH{{\mathbb H}}

\newcommand\bL{{\mathbb L}}

\newcommand\bV{{\mathbb V}}
\newcommand\bT{{\mathbb T}}

\newcommand{\E}{\mathbb{E}}

\newcommand{\codim}{\mathrm{codim}}

\renewcommand{\P}{\mathbb P}

\theoremstyle{plain}

\newtheorem{theorem}{Theorem}[section]

\newtheorem{proposition}[theorem]{Proposition}

\newtheorem{corollaire}[theorem]{Corollary}
\newtheorem{corollary}[theorem]{Corollary}

\theoremstyle{definition}
\newtheorem{definition}[theorem]{Definition}

\theoremstyle{remark}
\newtheorem{remarque}[theorem]{Remark}
\newtheorem{example}[theorem]{Example}

\usepackage[usenames,dvipsnames]{color}

\usepackage{cleveref}

\usepackage{amsmath,amsfonts,mathrsfs,amssymb}
\usepackage[letterpaper]{geometry}
\usepackage{setspace}
\usepackage{color}
\usepackage[all]{xy}
\usepackage{tikz-cd}
\usepackage{etoolbox}
\usepackage{enumitem}
\usepackage{fullpage}
\AtBeginEnvironment{quote}{\singlespacing\small}

\def\CC{\mathbb{C}}
\def\QQ{\mathbb{Q}}

\def\E{\mathcal{E}}

\def\ZZ{\mathbb{Z}}

\def\RR{\mathbb{R}}
\def\F{\mathcal{F}}

\theoremstyle{plain}

\author{Matt Kerr}
\address{Department of Mathematics, Washington University in St.~Louis, 1 Brookings Drive, St.~Louis, MO 63130, USA}
\email{matkerr@wustl.edu}

\author{Salim Tayou}
\address{Department of Mathematics, Dartmouth College, 29 N. Main Street, Hanover, NH 03755, USA}
\email{salim.tayou@dartmouth.edu}

\title{On the torsion locus of the Ceresa normal function}

\begin{document}

\begin{abstract}
We prove that the positive-dimensional part of the torsion locus of the Ceresa normal function in $\mathcal{M}_g$ is not Zariski dense when $g\geq 3$.  Moreover, it has only finitely many components with generic Mumford-Tate group equal to $\mathrm{GSp}_{2g}$; these components are defined over $\overline \Q$, and their union is closed under the action of $\mathrm{Gal}(\overline \Q/\Q)$. More generally, we study the distribution of the torsion locus of arbitrary admissible normal functions.  
\end{abstract}
\maketitle

\section{Introduction}

Let $g\geq 1$ and let $C$ be a smooth connected complex projective curve of genus $g$. A point $x\in C(\C)$ defines the Abel-Jacobi embedding $C\hookrightarrow J(C)$, where $J(C)$ is the Jacobian of $C$. Let $\iota\colon J(C)\rightarrow J(C)$ be the involution $y\mapsto -y$. Then the cycle $Z_C:=[C]-[\iota(C)]\in \mathrm{CH}^{g-1}(J(C))$ is a homologically trivial cycle on $J(C)$, and has a well defined image in $\mathrm{Griff}_1(J(C))$, the group of homologically trivial cycles modulo the group of algebraically trivial cycles (of dimension one). A famous theorem of Ceresa \cite{ceresa-annals} states that for $g\geq 3$, a very general smooth projective curve $C$ of genus $g$ has non-torsion Ceresa cycle in $\mathrm{Griff}_1(J(C))$. 

One can also consider the Abel-Jacobi class of $Z_C$ in the intermediate Jacobian associated to the primitive cohomology $V_C:=H^{2g-3}_{\text{pr}}(J(C),\Z(g-1))$.  Doing this over all of $\cM_g$ produces an admissible normal function $\nu\in \mathrm{ANF}_{\cM_g}(\cV)$, where $\cV$ is the VHS over $\cM_g$ with fiber $V_C$ over $[C]\in \cM_g$.  Informally, this is the graph of $[C]\mapsto \mathrm{AJ}_{J(C)}(Z_C)$.  Ceresa's proof shows that this normal function is nontorsion.

In this note, we prove the following theorem (see \S\ref{S4.1}).

\begin{theorem}\label{thm-cer}
Let $g\geq 3$ and let $H_g$ denote the hyperelliptic locus in $\mathcal{M}_g$. There exists a finite set of subvarieties $T_i$ of dimension $\leq 2g-2$, and a finite set of preimages of strict Shimura subvarieties $X_j$ of $\mathcal{A}_g$, such that if $U$ is the complement of $(\bigcup_{i}T_i)\cup (\bigcup_{j}X_j) \cup H_g$ in $\mathcal{M}_g$, then the locus in $U(\C)$ where $\nu$ is torsion is a countable set of points. Moreover, any positive dimensional component of the torsion locus of $\nu$ with Mumford-Tate group $\mathrm{GSp}_{2g}$ is defined over $\overline \Q$. 
\end{theorem}

Hain \cite{Hain-Ceresa} and Gao--Zhang \cite[Theorem 1.3 (ii)]{gao-zhang} have, 
independently and with different methods, also constructed a Zariski dense open subset $U$ of $\mathcal{M}_g$ such that there exists at most countably many points in $U(\C)$ where the Ceresa normal function is torsion.

\begin{remarque}
The locus of curves $C$ where $Z_C$ itself is torsion is contained in the torsion locus of $\nu$. In particular, by the standard spreading-out argument, the locus in $U$ of curves $C$ where $Z_C$ is torsion is a countable subset of $U(\overline{\QQ})$.  We note  that a similar result have also been independently obtained by S-W.~Zhang and Z. Gao \cite[Theorem 1.3-(i)]{gao-zhang}. 
\end{remarque}

Theorem \ref{thm-cer} is a particular instance of the following more general theorem which describes the distribution of the torsion locus of normal functions, see \Cref{main} for the most general version and \S\S\ref{S4.2}-\ref{S4.3} for applications to other normal functions.

\begin{theorem} \label{main-intro}
Let $\mathcal{M}$ be a smooth complex quasi-projective variety, and $\mathcal{V}$ be a non-isotrivial $\QQ$-simple $\ZZ$-VHS of negative weight on $\mathcal{M}^{\text{an}}$ with simple adjoint Mumford--Tate group $G$. Let $\phi\colon \mathcal{M}^{\text{an}}\to X=\Gamma\backslash D=\Gamma\backslash G_{\RR}/K$ denote the induced period map.  We assume that $D$ is a Hermitian symmetric domain 
and $X$ a connected Shimura variety, that $\phi$ is quasi-finite onto its image \textup{(}which is necessarily Hodge-generic in $X$\textup{)}, and that either $\cV\neq \cF^{-1}\cV$ or $\dim(\cM)\leq \tfrac{1}{2}\,\mathrm{rk}(\cV)$.

Let $\nu\in\mathrm{ANF}_{\mathcal{M}}(\mathcal{V})$ be a nontorsion admissible normal function, and write $\mathrm{TL}_{\text{pos}}(\nu)$ for the union of the \textup{(}countably many\textup{)} positive-dimensional components of its torsion locus.

Then $\mathcal{T}:=\overline{\mathrm{TL}_{\text{pos}}(\nu)}^{\text{Zar}}\subset \mathcal{M}$ is a proper algebraic subvariety.  Moreover, any component of $\mathcal{T}$ is either contained in the $\phi$-preimage of a proper locally symmetric subvariety of $X$, or is a component of $\mathrm{TL}_{\text{pos}}(\nu)$. 
\end{theorem}


Theorem 1.3 has also been independently obtained by Gao--Zhang, see \cite[Corollary 1.9]{gao-zhang}.

\begin{remarque}
If $\cV$ and $\nu$ arise from a family of varieties and a family of cycles thereon, with both (as well as $\cM$) defined over a number field, then these latter components of $\mathrm{TL}_{pos}(\nu)$ (not contained in the Hodge locus of $\cV$) are defined over $\overline{\QQ}$; see Theorem \ref{galois}.
\end{remarque}

We also prove a sufficient condition for the analytic density of the torsion locus, see \Cref{density} for the most general version of the following theorem.
\begin{theorem}\label{equidistribution-intro}
Assume that $\cV$ is a  $\QQ$-simple $\ZZ$-VHS of weight $-1$ and level\footnote{Recall that the \emph{level} of a Hodge structure $V$ (or variation thereof) is the difference of the largest and smallest $p$ for which $\mathrm{Gr}_F^p V\neq \{0\}$.} $1$ over $\mathcal{M}$, such that its period map is quasi-finite and its generic adjoint Mumford-Tate group is simple. We assume furthermore that $\cV$ has rank $2r\leq 2\dim(\mathcal{M})$. Let $\nu\in\mathrm{ANF}_{\mathcal{M}}(\mathcal{V})$ be a nontorsion admissible normal function. Then $\mathrm{TL}(\nu)
 $ is analytically dense in $\mathcal{M}$ and equidistributed with respect to the Betti form $\omega_{Betti}$, i.e., for every $U\subset \mathcal{M}$ relatively compact subset with zero measure boundary and for every differential form $\alpha$, we have:
 \[\int_{\{x\in U|\, n\cdot\nu(x)=0\}}\alpha \underset{n\rightarrow \infty}{\sim} \frac{n^{2r}}{\mathrm{vol}(\cJ(V_\R))}\cdot\int_{U}\alpha\wedge \omega_{Betti}^r~.\] 
\end{theorem}

\subsection{Ideas of the proofs}
The proofs of the previous theorems rely on arithmetic and functional transcendence results in Hodge theory, namely the Manin--Mumford conjecture and recent results on the study of the Hodge locus for variations of mixed Hodge structures proved by Baldi--Urbanik \cite{BU}. The equidistribution part in \Cref{equidistribution-intro} is inspired from the general equidistribution results established in \cite{tayoutholozan} and relies on the computation of the generic rank of normal functions with big monodromy established in \cite{gao-betti-map,gao-mixed-ax-schanuel} and which relies on the Ax--Schanuel Theorem for variations of mixed Hodge structures.

\subsection{Acknowledgement}
The present work started during the session on open problems in the conference ``The Ceresa Cycle in Arithmetic and Geometry'' at ICERM. We thank the organizers for their efforts and ICERM for its hospitality, Hain for an inspiring talk, and Ziyang Gao, Philip Griffiths, Jef Laga, Aaron Landesman, Ari Shnidman and Shou-Wu Zhang for valuable discussions and comments.  M.K.~was supported by NSF grant DMS-2101482, and S.T.~was supported by NSF grant DMS-2302388.



\section{Preliminaries on atypical Hodge loci and normal functions}
\subsection{Weak Hodge loci of variations of mixed Hodge structure}
Let $S$ be a smooth quasi-projective algebraic variety and let $\cE=\{\bE_{\ZZ},W_{\bullet}, \cF^{\bullet},\nabla\}$ be an admissible graded-polarizable variation of mixed Hodge structure over $S$.\footnote{Here $\bE_{\ZZ}$ is a torsion-free $\ZZ$-local system, the weight filtration $W_{\bullet}$ is defined on $\bE:=\bE_{\ZZ}\otimes_{\ZZ}\QQ$ and the Hodge filtration $\cF^{\bullet}$ is defined on $\cE:=\bE_{\ZZ}\otimes_{\ZZ}\mathcal{O}_S$. We make no notational distinction between the holomorphic vector bundle and its sheaf of sections (both are denoted by $\cE$).}  Let $(\cG,\cD_{\cG})$ be the generic mixed Hodge datum, where $\cG$ is the generic mixed Mumford-Tate group and $\cD_{\cG}$ is the mixed Mumford-Tate period domain. 
Let $(\cH,\cD_{\cH})$ be the monodromy datum of the variation:  here $\cH$ is the \emph{geometric monodromy group} of $\cE$, i.e.,~the neutral component of the $\QQ$-Zariski closure of the monodromy group $\cGamma$ associated to $\bE_{\ZZ}$.  Then $\cG$ and $\cH$ are algebraic groups over $\Q$ and by a well-known theorem of Deligne and Andr\'e \cite{Andre}, $\cH$ is a normal subgroup of $\cG^{der}$. Associated to this data, we have a holomorphic period map \[\Phi:S\rightarrow \cGamma\backslash \cD_{\cG}~.\] 
Let $(\overline{\cG},\overline{\cD}_{\overline{\cG}})$ denote the quotient mixed Hodge datum where $\overline{\cG}:=\cG/\cH$ and $\overline{\cD}_{\cG/\cH}$ is the mixed period domain associated to  $\overline{\cG}$. Then the monodromy period domain $\cD_{\cH}$ is isomorphic to a fiber of the projection $\cD_{\cG}\rightarrow \overline{\cD}_{\overline{\cG}}$ and the period map $\Phi$ factors, up to a finite \'etale cover, as $\Phi=\iota\circ \Phi_{\cH}$ where: 
\[S\xrightarrow{\Phi_{\cH}} \cGamma_{\cH}\backslash \cD_{\cH}\xhookrightarrow{\iota} \cGamma\backslash \cD_{\cG}\xrightarrow{\pi} \overline{\cGamma} \backslash \overline{\cD}_{\overline{\cG}}~.\]
We refer to $\Phi_{\cH}$ as the \emph{monodromy period map}.

Let $Y\subseteq S$ be a closed irreducible sub-variety of $S$. Then we have similarly a generic Hodge datum $(\cG_Y,\cD_{\cG_Y})$, monodromy datum $(\cH_Y,\cD^0_{Y})$, and quotient mixed datum $(\overline{\cG_Y},\overline{\cD}_{\overline{\cG_Y}})$. The restriction $\Phi|_Y$ of the period map $\Phi$ to $Y$ factors similarly, up to finite \'etale cover, as the composition of the first two arrows in
\[Y\rightarrow \cGamma_{\cH_Y}\backslash \cD^0_Y\hookrightarrow \cGamma_Y\backslash \cD_{\cG_Y}\xrightarrow{\pi_Y} \overline{\cGamma}_Y \backslash \overline{\cD}_{\overline{\cG_Y}}~.\]

\begin{definition}[\cite{klingler-otwinowska}]
Let $Y\subseteq S$ be a closed irreducible subvariety. We say that: 
\begin{enumerate}
    \item $Y$ has \emph{positive period dimension} if $\Phi(Y)$ is not zero dimensional.
    \item $Y$ is \emph{special} if $Y$ is maximal among closed subvarieties with Mumford-Tate group $\cG_Y$. Equivalently, $Y$ is a component of $\Phi^{-1}(\cGamma_Y\backslash \cD_{\cG_Y})$.
    \item $Y$ is \emph{weakly special} if $Y$ is maximal among closed subvarieties with monodromy group equal to $\cH_Y$. Equivalently, under the period map $\Phi$, $Y$ is a component of $\Phi^{-1}(\pi_Y^{-1}(\{x\}))$ where $x\in \overline{\cGamma}_Y\backslash\overline{\cD}_{\overline{\cG_Y}}$ is a Hodge-generic point.  
\end{enumerate}
\end{definition}



\begin{definition}
    Let $Y$ be a weakly special subvariety. We say that $Y$ is monodromically atypical if
    \[ \codim_{\Phi(S)}\left(\Phi(Y)\right)<\codim_{\cD_{\cH}}\cD^0_{Y}.\]
\end{definition}

Our main tool in this paper is the following result recently established by Baldi and Urbanik \cite[Theorem 7.1]{BU}:
\begin{theorem}\label{baldi-urbanik}
    There exists a finite set  $\Sigma$ of triples $(\cG_i,\cD_{i},\cN_i)_{i\in \Sigma}$, where $(\cG_i,\cD_{i})$ is a sub-Hodge datum of the generic Hodge datum $(\cG,\cD_{\cG})$, $\cN_i$ is a normal subgroup of $\cG_i$ whose reductive part is semisimple, and the following property holds: for each maximal monodromically atypical 
subvariety $Y\subset S$, there exist $i\in \Sigma$ and $y\in \cD_i$ for which $\cD_Y^0=\cN_i(\RR)^+\cN_i(\CC)^u\cdot y$ \textup{(}up to the action of $\cGamma$\textup{)}.
If in addition $Y$ has positive period dimension, then $\cN_i^{ad}$ is a nontrivial algebraic group.
\end{theorem}

\begin{remarque}
In the previous theorem, the statement that $\cD_0^Y=\cN_i(\RR)^+\cN_i(\CC)^u\cdot y$ implies (up to the action of $\cGamma$) that $\cN_i$ equals the geometric monodromy group of $\cE|_Y$, and $\Phi|_Y$ factors through $\cGamma_i\backslash\cD_i$,\footnote{Here one can take $\cGamma_i:=\cGamma\cap \cG_i$.} with image in the fiber of $\cGamma_i\backslash\cD_i\overset{\pi_i}{\to} \overline{\cGamma}_i\backslash \overline{\cD}_{\cG_i/\cN_i}$ over $\pi_i(y)$.  In particular, $\cG_i$ need \emph{not} equal the generic Mumford-Tate group of $\cE|_Y$ --- it merely \emph{contains} it, and normalizes the geometric monodromy group.  Hence if $\cN_i=\{1\}$, then $\cG_i$ is unconstrained by $Y$ and no ``finiteness'' result is achieved for subvarieties of period dimension zero.
\end{remarque}

\begin{remarque}
In this section we have used italicized letters for (Mumford-Tate and monodromy) groups associated to mixed variations.  We shall denote in the sequel by roman letters the groups attached to the associated gradeds (which are direct sums of pure variations), a passage denoted by $\cG\mapsto \mathrm{Gr}\,\cG$ in \cite{Andre}.  So for us $G$ denotes the reductive group $\mathrm{Gr}\,\cG$, and $\cG$ is an extension of $G$ by a unipotent group (which records the extension data).
\end{remarque}

\subsection{Brief review of normal functions}

An admissible normal function on $S$ with values in a polarizable VHS $\cV$ (of negative weight) is an extension of admissible\footnote{Admissibility is an asymptotic condition at the boundary $\overline{S}\backslash S$ which is always satisfied for VMHS arising from algebraic geometry.} variations of MHS on $S$ of the form
\begin{equation}\label{eq-NF}
0\to \cV\to \cE\to \ZZ(0)\to 0\,,\end{equation}
where $\ZZ(0)$ is the constant VHS of rank 1 and weight 0.  They constitute a group $\mathrm{ANF}_S(\cV)$ under Baer sum.

They may be viewed as functions in the following way.  There exist local lifts of $1\in \ZZ(0)$: to a holomorphic section $e_F$ of $\cF^0\cE$; and to a section $e_{\ZZ}$ of $\bE_{\ZZ}$. Their difference $e_{\ZZ}-e_F$ lifts to a local section of $\cV$, with ambiguity (arising from choices of $e_F,e_{\ZZ}$) in $\cF^0\cV+\bV_{\ZZ}$.  Passing to the generalized intermediate Jacobian bundle $\cJ(\cV)=\cV/(\cF^0\cV+\bV_{\ZZ})$, this construction produces a well-defined global holomorphic section $\nu\colon S\to \cJ(\cV)$.  (The group structure is given by adding sections.)  We write $\nu\in \mathrm{ANF}_S(\cV)$ and $\cE_{\nu}$ for the corresponding AVMHS.

By Griffiths transversality for $\cE_{\nu}$, a local lift $\tilde{\nu}$ of $\nu$ (to a section of $\cV$ itself) has derivative $\nabla\tilde{\nu}$ a local section of $ \cF^{-1}\cV\otimes\Omega^1_S$ (a property of $\nu$ known as \emph{quasi-horizontality}).  We can also analytically continue $\tilde{\nu}$ along a loop $\gamma\in \pi_1(S,s_0)$ and consider its monodromy $(\gamma-I)\nu\in V$ (the fiber of $\cV$ over $s_0$) which becomes well-defined in $V/(T_{\gamma}-I)V$ (with $T_{\gamma}$ given by the monodromy of the local system).  These notions of derivative and monodromy of $\nu$ are formalized by the \emph{infinitesimal} and \emph{topological invariants} $\delta\nu$ and $[\nu]$, defined as follows.

Consider the exact sequence of complexes of sheaves
\[0\to \cF^0\mathscr{C}^{\bullet}\oplus \bV_{\ZZ}\to \mathscr{C}^{\bullet}\to \frac{\mathscr{C}^{\bullet}}{\cF^0\mathscr{C}^{\bullet}\oplus \bV_{\ZZ}}\to 0\]
where $\mathscr{C}^{\bullet}:=(\cV\otimes \Omega^{\bullet}_S,\nabla)$, $\cF^0\mathscr{C}^{\bullet}:=(\cF^{-\bullet}\cV\otimes\Omega^{\bullet}_S,\nabla)$, and $\bV_{\ZZ}$ is placed in degree zero.  Write $\mathrm{NF}_S(\cV):=\bH^0(S,\tfrac{\mathscr{C}^{\bullet}}{\cF^0\mathscr{C}^{\bullet}\oplus \bV_{\ZZ}})$ for the group of holomorphic quasi-horizontal sections of $\cJ(\cV)$.

\begin{definition}
(i) In the composition
\[\xymatrix{\mathrm{ANF}_S(\cV)\ar @{^(->}[r] \ar [rrd]_{=:(\delta,[\cdot])}& \mathrm{NF}_S(\cV) \ar [r]^{\text{conn.~hom.}\mspace{50mu}} & \bH^1(\cF^0\mathscr{C}^{\bullet}\oplus \bV_{\ZZ}) \ar [d]^{\text{edge hom.}} \\ && \Gamma(S,\cH^1_{\nabla}(\cF^0\mathscr{C}^{\bullet}))\oplus H^1(S,\bV_{\ZZ})}\] $\delta\nu$ is the \emph{infinitesimal invariant} of $\nu\in \mathrm{ANF}_S(\cV)$, and $[\nu]$ the \emph{topological invariant} of $\nu$.

(ii) $\nu$ is said to be \emph{flat} if $\delta\nu=0$, and \emph{monodromically trivial} (resp.~\emph{torsion}) if $[\nu]$ is trivial (resp.~torsion).
\end{definition}

\begin{proposition}\label{prop_nf}
Consider the following conditions on $\nu$:
\begin{itemize}[leftmargin=1cm]
\item [\textup{(a)}] $\nu$ is torsion;
\item [\textup{(b)}] $\nu$ is flat;
\item [\textup{(c)}] $\nu$ is monodromically torsion;
\item [\textup{(d)}] $\bE_{\nu}$ is split \textup{(}there is a lift of $1\in \ZZ(0)$ to $\bE_{\nu}$\textup{)} after multiplying $\nu$ by a nonzero integer or passing to a finite subgroup of $\pi_1(S)$;
\item [\textup{(e)}] the geometric monodromy group $\cH$ of $\cE_{\nu}$ is semisimple \textup{(}has trivial unipotent radical\textup{)}, i.e.,~equals the geometric monodromy group $H$ of $\cV$.
\end{itemize}
In general, \textup{(a)} $\implies$ \textup{(b,c,d,e)}, and \textup{(c,d,e)} are all equivalent.  If the ``fixed part'' $H^0(S,\bV)=\{0\}$, then \textup{(a)} and \textup{(c,d,e)} are equivalent.  If, in addition, the sheaf map $\nabla\colon \cF^0\cV\to \cF^{-1}\cV\otimes \Omega^1_S$ is everywhere injective, then all conditions are equivalent.
\end{proposition}
\begin{proof}[Sketch]
Working $\otimes\QQ$, if $[\nu]=0$ then $\bE_{\nu}=\bV\oplus \QQ$.  If also $H^0(S,\bV)=\{0\}$, then $H^0(S,\bE_{\nu})=\QQ$ and the Theorem of the Fixed Part \cite[(4.19)]{Steenbrink-Zucker} says the latter underlies a (constant) sub-AVMHS of $\cE_{\nu}$.  Since it is of rank one, it can only be of type $(0,0)$, splitting \eqref{eq-NF} and rendering $\nu=0$.  Thus $[\cdot]$ is injective, whence (c)$\implies$(a), when $\bV$ has trivial fixed part.  The additional condition required for (b)$\implies$(a) is to make the edge homomorphism injective.

For the equivalence of (c) and (d), tracing through the definition shows that $[\nu]$ is the image of $1\in H^0(S,\ZZ(0))$ under the connecting homomorphism for the extension 
\begin{equation}\label{eq-LS}
0\to \bV_{\ZZ}\to \bE_{\nu,\ZZ}\to \ZZ(0)\to 0
\end{equation}
of local systems.  Finally, $\cH$ is an extension of $\mathrm{Gr}\,\cH$ by its unipotent radical $U$; (e) says that $U=\{1\}$, and (d) is the same as $\cH=\mathrm{Gr}\,\cH$, so (d)$\iff$(e).
\end{proof}

\begin{definition}\label{defn-gen}
A \emph{projection} of $\nu$ is its image under $\mathrm{ANF}_S(\cV)\to \mathrm{ANF}_S(\cV')$ for any \emph{nontrivial} surjection $\cV\twoheadrightarrow \cV'$ of VHS.  We say that $\nu$ is \emph{generic} if no projection of $\nu$ has torsion topological invariant.
\end{definition}

\begin{corollaire}\label{cor-NF}
\textup{(i)} $\nu$ is generic if and only if the unipotent radical of $\cH$ \textup{(}the geometric monodromy group of $\cE_{\nu}$\textup{)} is isomorphic to $V$.

\textup{(ii)} If $\cV$ has no constant factors, $\nu$ is generic if, and only if, every projection of $\nu$ is nontorsion.

\textup{(iii)} If $\cV$ is irreducible and non-constant, $\nu$ is generic if, and only if, it is nontorsion.
\end{corollaire}
\begin{proof}
If $\cV$ is simple, then non-semisimplicity of $\cH$ is the same as having $U=V$.  More generally, applying this to projections with $\cV'$ simple and using the equivalence of (c) and (e) in the Proposition shows that genericity is equivalent to $U$ projecting to each $V'$, i.e.,~$U=V$.  For (ii-iii), the fixed part is trivial and we use the resulting equivalence of (a) and (c).
\end{proof}

\begin{remarque}\label{rem-NF}
In the geometric context, normal functions graph Abel-Jacobi mappings in families of cycles and higher cycles.  Given a smooth projective family $\cX\to S$, with $\cV\subset \cH^n(\cX/S)(m)$ and $r=2m-n-1\geq 0$, a primitive\footnote{This means that the restrictions $\mathfrak{Z}|_{X_s}$ to fibers of $\cX\to S$ are homologically trivial, which is a nontrivial condition only for $r=0$. Also note that $\cV$ has negative weight $n-2m=-r-1$.} cycle $\mathfrak{Z}\in \mathrm{CH}^m(\cX,r)\cong_{\otimes \QQ}\mathrm{Gr}_{\gamma}^m K_{r}^{\text{alg}}(\cX)$ produces a normal function $\nu_{\mathfrak{Z}}\in \mathrm{ANF}_S(\cV)$ via \[\nu(s):=\mathrm{AJ}_{X_s}(\mathfrak{Z}|_{X_s})\in \cJ(V_s)=\frac{V_{s,\CC}}{F^0V_{s,\CC}+\bV_{\ZZ}}\cong \mathrm{Ext}^1_{\text{MHS}}(\ZZ(0),V_s)\,.\]
When $r=0$, $\cJ(V_s)$ is a compact complex torus known as an intermediate Jacobian; it is usually\footnote{Exceptions here would include the case where $V_s$ is a CM Hodge structure.} not algebraic if the level of $V_s$ is larger than one.  When $r>0$, it is a noncompact complex torus.
\end{remarque}

By the main theorem of \cite{BP}, the zero locus of $\nu$ is an algebraic subvariety of $S$.  Since the zero-locus of $m\nu$ is the $m$-torsion locus of $\nu$, we get that the torsion locus with torsion order bounded (say, by $N$) is also algebraic.  Thus the torsion locus of $\nu$ is a countable union of algebraic subvarieties.  In the remainder of this paper we aim to improve this statement.

\section{Torsion loci of admissible normal functions}

\subsection{The main theorem}
Let $\mathcal{M}$ be a smooth quasi-projective complex variety and let $\cV=\{\bV_\Z,\cF^\bullet,\nabla\}$ be a non-trivial polarized $\Z$-VHS of negative weight on $\mathcal{M}$. Let $G$ denote the generic Mumford-Tate group of $\cV$ and $\phi\colon \mathcal{M}_{\CC}^{\text{an}}\to X=\Gamma\backslash D=\Gamma\backslash G_{\RR}/K$ be the induced period map.  By construction the image $M:=\phi(\cM)$ is Hodge-generic in $X$; while $X$ is in general not an algebraic variety, we know that $M$ is one by \cite{bbt}.


Let $(H,\cD_H)$ be the monodromy data of $\cV$ and let $\phi_H:\cM\rightarrow X_H:=\Gamma_H\backslash \cD_{H}$ be the monodromy period map.
\begin{definition}
A subvariety $Y\subset \cM$ is said to be of \emph{factorwise positive dimension} with respect to $\phi$, if the projection of $\phi(Y)$ to each simple factor of $X_H$ has positive dimension.\footnote{One may need to take a finite \'etale cover to define these projections.}    
\end{definition}

Let $\cV_0$ be the largest isotrivial summand of $\cV$. Then the variation $\cV$ admits a decomposition: 
\begin{align}\label{decomposition}
\cV=\cV_0\oplus \bigoplus_{\ell\in I} \cV_\ell
\end{align}
where for $\ell\in I$, $\cV_\ell$ are $\Q$-simple non-isotrivial variations. Let $\nu\in\mathrm{ANF}_{\mathcal{M}}(\mathcal{V})$ be a nontorsion admissible normal function. Then $\nu$ admits a decomposition $\nu=(\nu_\ell)_\ell$ which reflects the decomposition in \eqref{decomposition}. Let $V$ be the $\Q$-vector space underlying a fiber of $\cV$, which is also equal to $H^0(\widetilde{\cM},\iota^*\bV_{\CC})$, where $\iota:\widetilde{\cM}\rightarrow \cM$ is the universal cover. 

We can now state the main theorem of this paper. 
\begin{theorem}\label{main} 
Let $\nu\in \mathrm{ANF}
_{\mathcal{M}}(\cV)$ be a generic admissible normal function on $\mathcal{M}$. Let $\mathrm{TL}_{f-pos}(\nu)$ be the \textup{(}countable\textup{)} union of components of the torsion locus of $\nu$
which are of factorwise positive dimension with respect to $\phi$. Assume that at least one of the following conditions holds:
\begin{itemize}[leftmargin=1cm]
    \item [\textup{(i)}] the vector bundle $\cV/\cF^{-1}\cV$ is non-trivial; or
    \item [\textup{(ii)}] $\dim(\mathcal{M})\leq \mathrm{rk}(\cJ(\cV))$.
\end{itemize}
Then 
$\mathcal{T}:=\overline{\mathrm{TL}_{f-pos}(\nu)}^{Zar}$ is a strict algebraic subvariety of $\mathcal{M}$. Moreover, there exist finitely many projections $(\nu_i)_i$ of $\nu$ such that each component of $\mathcal{T}$ is either contained in a component of the Hodge locus for $\phi$, or contained in $\mathrm{TL}_{f-pos}(\nu_i)$.
\end{theorem}

\begin{remarque}
If $\cV$ is in weight $-1$ (for ``usual'' normal functions), then $\cV$ has odd level and (i) says this level is $>1$; if $\cV$ has weight $-2$, (i) says that $\cV$ has level $>0$; and if $\cV$ has weight $<-2$, (i) holds automatically.  Condition (ii) is thus of interest only when $\cV$ has weight $-1$ level $1$ (where $\mathrm{rk}(\cJ(\cV))=\tfrac{1}{2}\dim_{\QQ}(V)$), or weight $-2$ level $0$ (where $\mathrm{rk}(\cJ(\cV))=\dim_{\QQ}(V)$).
\end{remarque}

\begin{proof}
We will regard $\nu$ as a variation of mixed Hodge structure $\mathcal{E}_{\nu}\in \mathrm{Ext}^1_{\text{AVMHS}(\mathcal{M})}(\ZZ(0),\cV)$.  Since $\nu$ is generic, the geometric monodromy group $\cH$ of $\E_{\nu}$ is an extension of $H$ by $V$ (by Corollary \ref{cor-NF}(i)); and thus the Mumford-Tate group $\cG$ of $\cE_{\nu}$ is an extension of $G$ by $V$.  Consider the period map for $\cE_\nu$: 
\[\Phi\colon \mathcal{M}\to \cX=\cGamma\backslash \mathcal{D}=\cGamma\backslash\mathcal{G}(\R)/\mathcal{K}~,\]  
where the arithmetic group $\tilde{\Gamma}$ is an extension of $\Gamma$ by the translation group $V_{\ZZ}$. Consider also the monodromy period map \[\Phi_\cH\colon \mathcal{M}\to \cX_{\mathcal{H}}=\cGamma_\cH\backslash \mathcal{D}_\cH=\cGamma\backslash\mathcal{H}(\R)/\mathcal{K}_{\cH}~.\]  
The $\QQ$-vector space $V$ admits an $H$-stable decomposition into irreducible factors $V=V_0\oplus\bigoplus_{\ell\in I} V_\ell$ which reflects \eqref{decomposition}, and where $V_0=V^{H}$.

There is a natural projection $\pi\colon \cX\to X$ with fibers generalized intermediate Jacobians $\{\cJ_x\}$, as well as natural sections corresponding to torsion points in $V_{\CC}/V_{\ZZ}$ preserved by $\Gamma$. This projection is compatible with the decomposition from \eqref{decomposition}, since the intermediate Jacobian $\{\cJ_x\}$ decomposes as the product of the intermediate Jacobians associated to each summand in \eqref{decomposition}. This results in a decomposition $\pi\colon\cD=\cD_0\times_{D} \prod_{\ell\in I}\cD_\ell\rightarrow D$ and $\pi\colon\cX=\cX_0\times_X\prod_{\ell\in I}\cX_\ell\rightarrow X$. Here the fiber of $\cD_{\ell}\to D$ over $\varphi\in D$ is $V_{\ell,\CC}/F^0_{\varphi}V_{\ell,\CC}$, and that of $\cX_{\ell}\to X$ is the generalized  Jacobian $\cJ(V_{\ell})$. The monodromy period map factors furthermore as:
\[\Phi_\cH:\mathcal{M}\rightarrow \cJ(V_0)\times \prod_{\ell\in I}\cX_\ell~,\] where $\cJ(V_0)$ is the generalized Jacobian associated to the isotrivial variation $\cV_0$.\footnote{One may need to take a finite \'etale cover to define this map, which we assume already done.  All products ``$\prod$'' are relative (to $D$ or to $X$).}

Let $Y\subset \mathcal{M}$ be an irreducible subvariety, and denote by $(\mathcal{G}_{Y},\mathcal{D}_{Y})$ the mixed Hodge data associated to $\mathcal{E}_{\nu}|_{Y}$.  Further, we write $\mathcal{H}_Y\trianglelefteq \mathcal{G}_Y^{\text{der}}$ and $\mathcal{D}_Y^0\subseteq \mathcal{D}_Y$ for the geometric monodromy group and the monodromy period domain, which is the orbit of a basepoint\footnote{that is, a point whose image in $\tilde{X}$ is contained in the image of $\Phi|_Y$} in $\mathcal{D}_Y$ by $\mathcal{H}_Y(\RR)^+\mathcal{H}^{\text{un}}_Y(\CC)$. 
Recall that $Y$ is monodromically atypical for $\Phi$ if $Y$ is weakly special and $\text{codim}_{\mathcal{M}}(Y)<\text{codim}_{\mathcal{D}_\cH}(\mathcal{D}^0_Y)$.

Now let $\{(\cG_i,\cD_i,\cN_i)\}_{i\in \Sigma}$ be the finite set of triples guaranteed for $\cE_{\nu}$ by \Cref{baldi-urbanik}.  Take $Y\subset \cM$ to be a \emph{maximal} monodromically atypical subvariety, i.e.,~monodromically atypical and not contained in a larger monodromically atypical subvariety.  Then up to the action by $\cGamma$,\footnote{In fact, we can ignore the action by $\Gamma$ here, since this doesn't affect images in $\cX$.}
there exist $i\in \Sigma$ and $y\in \cD_i$  such that $\cH_Y=\cN_i$ and $\mathcal{D}_Y^0=\mathcal{N}_i(\R)\mathcal{N}_{i}^{u}(\C)\cdot y$.

We consider the various possibilities for this triple.  For $i\in \Sigma$ belonging to some subset (say, $\Sigma_0\subset \Sigma$), $G_i:=\mathcal{G}_i/\mathcal{G}_i\cap V$ has $G_i^{der}< G^{der}$ a proper subgroup.  In this case, $D_Y:=\pi(\mathcal{D}_Y)$ is a proper subdomain of $D$ and $Y$ is contained in a component $Z$ of the Hodge locus for $\phi$ (i.e., for $\cV$).  More precisely, there is a Mumford-Tate subdomain $D_i:=G_{i,\RR}/(K\cap G_{i,\RR})\subset D$ and $Z\subset \cM$ is one of finitely many components of the preimage of $\Gamma_i\backslash D_i\subset X$.  The union of these $Z$ as $i$ ranges over $\Sigma_0$ is a proper algebraic subvariety $\mathscr{Z}\subset \cM$.

If, on the other hand, $G_i^{der}=G^{der}$, then we assume furthermore that $Y$ has factorwise positive dimension for $\phi$. Since ${N}_i:=\cN_i/(\cN_i\cap V)$ is normal in $G^{der}$ and equals $H_Y$, and $H_Y=H$ by factorwise positivity, we have $N_i=H_Y=H$. Since $\mathcal{N}_i\cap V$ is $N_i$-stable, it is of the form\footnote{If $V$ has repeated factors, one may need to rearrange the decomposition of isotypical components.  However, such rearrangements are still subject to the finiteness of $\Sigma$.} \[V'_{0}\oplus\bigoplus_{\ell\in I'}V_\ell\subsetneq V_{0}\oplus\bigoplus_{\ell\in I}V_\ell~,\] for a subset $I'\subset I$ and a $\Q$ sub-Hodge structure $V'_0\subset V_0$. Hence we conclude that $\cH_Y$ is, up to an irrelevant torus part, an extension of $H$ by $V'_{0}\oplus\bigoplus_{\ell\in I'}V_\ell$. 

The group $\mathcal{G}_i$ is on the other hand an extension of $G$ by $W_0\oplus \bigoplus_{\ell\in I_1}V_\ell$ for some $I_1\subseteq I$ $W_0\subset V_0$ a sub Hodge structure. Since $\cH_Y=\cN_i$ is a subgroup of $\mathcal{G}_i^{der}$, $I'\subseteq I_1$ and $V'_0\subseteq W_0$. Let $U=W_0/V'_0\oplus \bigoplus_{\ell \in I_1\backslash I'}V_\ell$. Then $U$ and $\cH_Y$ are both normal subgroups of $\mathcal{G}_i$ with $\mathcal{G}_i^{der}/\cH_Y\simeq U$, whence $\mathcal{G}_i^{der}$ is an almost direct product of $\cH_Y$ and $U$.  This implies that $H_Y=H$ acts trivially on $U$, and hence that $U\subset V_0$ as a sub-Hodge structure. Thus $I_1=I'$.  

Let $\cD_{W_0}\rightarrow D$, $\cD_{V_0/W_0}\overset{\pi_0}{\to} D$, and $\cD_{\ell}\overset{\pi_{\ell}}{\to} D$ be the pre-intermediate-Jacobian fibrations associated to $W_0$, $V_0/W_0$, and $V_{\ell}$. Then there are torsion sections $\sigma_i^{(0)}$ of $\pi_0$, and $\sigma_i^{(\ell)}$ of $\pi_{\ell}$ for $\ell\in J:=I_1^c$, such that
\[\cD_i=\cD_{W_0}\times_D \sigma_i^{(0)} (D_{V_0/W_0})\times_D \prod_{\ell \in I_1}\cD_\ell\times_D \prod_{\ell\in J}\sigma_i^{(\ell)}(D_\ell)\,.\] 
Writing $\cX_i$ for the image of $\cD_i$ in $\cX$, $\Phi^{-1}(\cX_i)\subset\cM$ is a union of (finitely many) components of the torsion locus of the projection $\nu_i=\nu_{V_0/W_0}+\sum_{\ell\in J}\nu_\ell$ of our normal function $\nu$. By the finiteness result of \Cref{baldi-urbanik}, only finitely many such projections $\nu_i$ and sections $\sigma_i=\{\sigma_i^{(0)},\{\sigma_i^{(\ell)}\}\}$ appear (up to translation by $V_{\ZZ}$).  In particular, we can bound the torsion order of $\nu_i$ uniformly in $i$, say by $m$.  

As for $\cD_Y^{0}$, it is determined up to the choice of a point $y\in \cJ(U)$, the intermediate Jacobian of the constant Hodge structure $U$.\footnote{Here $U$ is constant possibly after passing to a finite \'etale cover of $\cM$.} As we will be primarily interested in the situation where $y$ is torsion, we make that assumption from now on. In other words, we suppose that $Y$ is special. Notice that \Cref{baldi-urbanik} does not say anything about the boundedness of the torsion orders of such points, but we are going to prove that this torsion order is indeed bounded. 

Notice first that if either $J$ is nonempty or $V_0/W_0\neq \{0\}$, then $Y$ is contained in the strict subvariety where $m\nu_i$ vanishes and we can skip the next discussion. Otherwise, assume that $V_0=W_0$ and $J=\emptyset$, and thus that $\cG_i=\cG$. Denote the projection of the monodromy period map to $\cJ(U)$ by $\Phi_0:\mathcal{M}\rightarrow \cJ(U)$. Since $Y$ is assumed to be (monodromically) atypical, and the only ``drop'' from $\cG$ to $\cH_Y=\cN_i$ is $U$, $Y$ must be an atypical fiber of $\Phi_0$ over a torsion point.

Suppose $U$ has an irreducible sub-Hodge structure $U'$ which is not of weight -1 level 1 or weight -2 level 0.  The topological invariant associated to the projection $\nu_{U'}$ is computed by the composition $\pi_1(\cM)\twoheadrightarrow H_1(\cM,\ZZ)\to H_1(\cJ(U'),\ZZ)\cong U'_{\ZZ}$ induced by $\Phi_0'$, the Zariski closure of whose image is (the underlying $\QQ$-vector subspace of) a sub-HS $U''\subset U'$.  On the other hand, we have $F^{-1}U_{\CC}'\subsetneq U_{\CC}'$; and the differential $d\Phi_0'\colon T\cM\to T\cJ(U')\cong U_{\CC}'/F^0U_{\CC}'$ factors through $F^{-1}U_{\CC}'/F^0U_{\CC}'$ by Griffiths transversality, which forces $U''_{\CC}\subset F^{-1}U'_{\CC}$.  Since $U'$ is irreducible, this forces $U''=0$, in contradiction to the assumed genericity of $\nu$.

Hence we may assume that all the Hodge structures appearing in $U$ have either weight -1 level 1 or weight -2 level 0. Then $\cJ(U)$ is a semi-abelian scheme, and it is an algebraic variety. Assume for the sake of contradiction that the image of $\Phi_0$ contains infinitely many torsion points, whose preimages are atypical. By \cite{bbt-mixed}, the image of $\Phi_0$ is an algebraic subvariety and by the solution to the Manin-Mumford conjecture \cite{manin-mumford, pila-zannier}, the closure of the torsion locus in ${\Phi}_0(\mathcal M)$ is a translate of a subgroup scheme $B$. If $B$ is strict, then ${\Phi}^{-1}_0(B)$ is a (monodromically) atypical subvariety of $\mathcal{M}$ that strictly contains $Y$, which contradicts the maximality of $Y$. Hence $B=\cJ(U)$ and $\Phi_0$ is dominant. But since (by atypicality of $Y$) the fibers of $\Phi_0$ have dimension  $>\dim(\mathcal{M})-\dim \cJ(U)$ over a Zariski dense locus in $B$, this implies that $\Phi_0$ has relative dimension $>\dim(\mathcal{M})-\dim \cJ(U)$, and as it is dominant, we get a contradiction. Thus $\Phi_0(\cM)$ contains only finitely many torsion points and we can bound the torsion order in the factor $\cJ(U)$.  Conclude that $Y$ is a component of the torsion locus of a normal function $\nu_i=\nu_{W_0/V'_0}$ and where the possible torsion orders are bounded (say, by $m$ again). 

We have now produced the finite collection $(\nu_i)_i$ of projections of $\nu$ in the statement of the Theorem.  Let $\mathscr{T}\subset \cM$ denote the union of the zero loci of $m\nu_i$.  By \cite{BP} and our genericity assumption on $\nu$, this is a proper algebraic subvariety.

Finally, let $T$ be a factorwise positive-dimensional component of $\mathrm{TL}(\nu)$.  On $T$, the monodromy period map factors through some torsion section $\sigma\colon D_{\cH}\to \cD_{\cH}$ of codimension $\mathrm{rk}(\cJ(\cV))$.  Now if (i) holds, then (invoking Griffiths transversality once more) the rank of $\delta\tilde{\nu}|_p\colon T_p\cM\to \cF^{-1}V_p/\cF^0V_p$ at a point $p\in T$, and hence $\text{codim}_{\cM}(T)$, is strictly less than $\mathrm{rk}(\cJ(\cV))$.  If (ii) holds, then $\text{codim}_{\cM}(T)<\dim(\cM)\leq \mathrm{rk}(\cJ(\cV))$.  So in either case $\text{codim}_{\cM}(T)<\mathrm{rk}(\cJ(\cV))\leq$ monodromy codimension of $T$, which makes $T$ monodromically atypical for $\cE_{\nu}$.  Let $Y$ be a maximal monodromically atypical subvariety containing $T$, and assume $Y$ is not contained in $\mathscr{Z}$.  Then the analysis above shows that $\cD_Y^0$ is determined by a torsion point, since it contains $T$, and that there exists $i$ such that $Y$ is a component of $\mathrm{TL}(\nu_i)$ with torsion order bounded by $m$.  That is, $Y\subset \mathscr{T}$.
\end{proof}



We now state some corollaries that follow from the previous theorem.

\begin{corollary}\label{corollary:non-density}
    If $\overline{\mathrm{TL}_{f-pos}(\nu)}^{Zar}$ is Zariski dense in $\mathcal{M}$, then $\cV/\cF^{-1}\cV$ is trivial and $\cV$ has rank less than  $2\dim(\mathcal{M})$.
\end{corollary}

If the variation $\cV$ is $\Q$-simple and non-isotrivial, then notice that in the proof of \Cref{main}, the only possibility for $\nu_i$ is $\nu$, and if $\nu$ is generically non-torsion, then it is generic by Corollary \ref{cor-NF}(iii). Hence we get the following result. 

\begin{corollary} 
Assume that $\cV$ is $\Q$-simple and non-isotrivial of negative weight with quasi-finite period map. Let $\nu\in\mathrm{ANF}_\mathcal{M}(\cV)$ be a non-torsion admissible normal function. Assume that either the vector bundle $\cV/\cF^{-1}\cV$ is non-trivial or the rank of $\mathcal{V}/F^0\mathcal{V}$ is at least\footnote{It is worth remarking here that if this rank \emph{exceeds} $\dim(\cM)$, then one \emph{expects} the non-special zero-dimensional part of the torsion locus to be finite.} $\dim(\cM)$.




Then $\mathcal{T}:=\overline{\mathrm{TL}_{\text{pos}}(\nu)}^{\text{Zar}}\subset \mathcal{M}$ is a proper algebraic subvariety.  Moreover, any component of $\mathcal{T}$ is either contained in the preimage under $\phi$ of a proper locally symmetric subvariety of $X$, or is a component of $\mathrm{TL}_{\text{pos}}(\nu)$.
\end{corollary}

\begin{remarque}
The previous theorem shows there are at most finitely many Shimura subvarieties $X_i\subsetneq X$ such that the normal function is generically non-torsion on $\phi^{-1}(X_i)$ but is torsion on a (possibly dense) countable union of subvarieties thereof.  As an immediate consequence, at most finitely many Hecke translates of such an $X_i$ will support dense positive-dimensional torsion locus of $\nu$.    
\end{remarque}



\subsection{Analytic density and equidistribution of the torsion locus}

We assume in this section that $\cV$ has weight $-1$ and level $1$, that all its $\Q$-factors are non-isotrivial, and $H=G^{der}$. We continue to assume that the period map for $\cV$ is quasi-finite. Let $\nu\in\mathrm{ANF}_{\mathcal{M}}(\cV)$ be a generic normal function in the sense of Definition \ref{defn-gen}.

To analyze the density properties of the torsion locus of $\nu$, we need to understand  better the local properties of the normal function $\nu$, namely its \emph{generic rank} which we will now define. 

Let $\cJ(\cV)\rightarrow \mathcal{M}$ be the family of intermediate Jacobians of which $\nu$ is a section. For $U\subset \mathcal{M}$ an open simply connected subset, the family of Jacobians $\cJ(\cV)\rightarrow U$ admits a $C^{\infty}$-trivialization $\cJ(V_\R)\times U\rightarrow U$ where $\cJ(V_\R)$ is a compact torus. Projecting to the first factor and lifting to $V_\R$ yields the \emph{Betti map} $\nu_\R:U\rightarrow V_\R$. 

For $s\in U$, we can define the rank of $\nu$ at $s$ as the rank of the differential of $\nu_\R$ at $s$. Denoting by $\widetilde{\mathcal{M}}$ the universal cover of $\mathcal{M}$, the previous discussion produces a map $\nu_\R:\widetilde{\mathcal{M}}\rightarrow V_\R$. The fibers of $\nu_\R$ are complex analytic subvarieties of $\widetilde{\cM}$.
\begin{definition}
The generic rank of $\nu$ is $\max_{s\in \widetilde{\mathcal{M}}}\rank_s(d\nu_\R)$. It is an even integer.
\end{definition}

In our setting, the rank of the Betti map was fully studied in \cite{gao-betti-map}, and a key ingredient was to use the mixed Ax-Schanuel Theorem. The following result is a particular case of \cite[Theorem 1.4. (ii)]{gao-betti-map}. Using this idea, we reproduce this result in a Hodge theoretic setting to make our paper more self-contained. The proof aligns more closely with \cite[Theorem 9.1]{gao-mixed-ax-schanuel}. We refer also to \cite[Theorem 3.1]{gao-zhang} for a more general result in weight $-1$ and to \cite{khelifa-urbanik} for a purely Hodge theoretic approach. 
\begin{theorem}\label{generic-rank-betti}
Assume that $G^{ad}$ is simple and $\cV$ has no isotrivial factors. Then the generic rank of $\nu$ is equal to $\mu:=\min(\dim_{\R}{\mathcal{M}},\dim_\Q V)$.
\end{theorem}
\begin{proof}
Let $d$ be the complex dimension of $\mathcal{M}$. Obviously the generic rank $2k$ of $\nu$ cannot be larger than $\mu$.  By way of contradiction, assume that this rank is smaller than $\mu$ (hence that $k<d$ and $k<\tfrac{1}{2}\dim_{\QQ}(V)$). Let $\widetilde{\Phi}:\widetilde{\mathcal{M}}\rightarrow \cD$ be the lift of the period map associated to the variation of mixed Hodge structure $\cE_\nu$. Every point $x\in \cJ(V_\R)$ defines a holomorphic section $\sigma_x:D\rightarrow \cD$ and it is clear that $\tilde{\Phi}^{-1}(\sigma_x(D)\cap \widetilde{\Phi}(\widetilde{\cM}))=\nu_\R^{-1}(\{x\})$. 

Since $\nu_\R$ has generic rank $2k$, we can find an analytic open subset $U\subset \widetilde{\mathcal{M}}$ such that the fibers of $\nu_\R$ above $\nu_\R(U)$ are complex analytic of complex dimension $d-k$. Furthermore, since the Hodge locus of $\widetilde{\mathcal{M}}$ is a countable union of strict analytic subvarieties, we can find a point $y\in U$ which is Hodge generic by Baire's theorem. Let $x=\nu_\R(y)\in V_\R$; then it satisfies the following properties: 
\begin{enumerate}
    \item the fiber $Y_0=\nu_\R^{-1}(x)$ is complex analytic of complex dimension $d-k$ ($>0$ by the assumption on $k$);
    \item $Y_0$ is Hodge generic (for $\cE_{\nu}$) inside $\widetilde{\mathcal{M}}$.
\end{enumerate}

Let $\Delta=\cD\times_{\Gamma\backslash \cD}\mathcal{M}$ and $W=D_x\times \mathcal{M}$. Then $Y_0$ is a complex analytic component of the intersection $W\cap \Delta$ in $\cD\times\mathcal{M}$ and we have: 
\[\codim_{\cD\times\mathcal{{M}}}(Y_0)=\dim(\cD)+k,\mspace{20mu} \codim_{\cD\times\mathcal{{M}}}(W)=\frac{1}{2}\dim_\Q V,\mspace{20mu}\codim_{\cD\times\mathcal{{M}}}(\Delta)=\dim(\cD)\,.\]
Hence (by our assumption on $k$) the following inequality is satisfied: 
\[\codim_{\cD\times\mathcal{{M}}}(Y_0)<\codim_{\cD\times\mathcal{{M}}}(W)+\codim_{\cD\times\mathcal{{M}}}(\Delta)~.\]

By the Ax--Schanuel theorem for variations of mixed Hodge structures \cite{Klingler-AS-mixed}, we conclude that the projection of $Y_0$ to $\cM$ is contained in a \emph{strict} weakly special subvariety $\cZ$.  That is, the period map sends $\cZ$ into (a quotient of) a \emph{proper} weak Mumford-Tate subdomain of $\cD$, and so the geometric monodromy group $\cH_{\cZ}$ of $\cE_{\nu}|_{\cZ}$ is a \emph{proper} subgroup of $\cG^{der}$.  But since $Y_0$ is Hodge generic, the generic Mumford-Tate group of $\cE_{\nu}|_{\cZ}$ is equal to $\cG$; and because $Y_0$ is positive-dimensional (and $G^{ad}$ is simple), we have that $H_\cZ=G^{der}$. As $\mathcal{H}_\cZ$ is a normal subgroup of $\cG$ and $G^{der}$ acts non-trivially on every factor of $V$, we conclude that $\cH_\cZ=\cH=\cG^{der}$, which is a contradiction. This proves the desired result. 
\end{proof}

Let $\psi:V_\Z\times V_\Z\rightarrow \Z$ be the polarization form and let $\beta$ be the corresponding harmonic $(1,1)$ form on $\cJ(V_\R)$. Then $\beta^r$ determines a volume form on $\cJ(V_\R)$, where $r=\frac{1}{2}\dim_\Q V$ and let $\mathrm{vol}(\cJ(V_\R))$ be the volume of $\cJ(V_\R)$ with respect to $\beta^r$, which is also the covolume of the lattice $V_\Z$ in $V_\R$\footnote{The quadratic form on $V_\R$ is $\psi(\cdot,J \cdot)$ where $J$ is the almost complex structure on $V_\R$.}. It is well known that the torsion points on $\cJ(V_\R)$ are equidistributed with respect to $\beta^r$; that is, for every open subset $U\subset \cJ(V_\R)$ with measure-zero boundary, we have 
\[\left|\{z\in U|\, n\cdot z=0\}\right|\underset{n\rightarrow \infty}{\sim}\frac{n^{2r}}{ \mathrm{vol}(\cJ(V_\R))}\cdot \int_{U}\beta^r~.\]
The pullback form $\omega_{Betti}=\nu^{*}\beta$ to $\widetilde{\mathcal{M}}$ is translation invariant and descends to a form on $\mathcal{M}$, also known as the \emph{Betti form}.

\begin{theorem}\label{density}
We keep the same assumptions as \Cref{generic-rank-betti} and we assume furthermore that $\cV$ has rank $2r\leq 2\dim(\mathcal{M})=2d$. Then $\mathrm{TL}(\nu)$ is analytically dense in $\mathcal{M}$ and equidistributed with respect to the Betti form $\omega_{Betti}$, i.e., for every $U\subset \mathcal{M}$ relatively compact subset with zero measure boundary and differential form $\alpha$, we have:
\begin{equation}\label{eq-equi}
\int_{\{x\in U|\, n\cdot\nu(x)=0\}}\alpha \underset{n\rightarrow \infty}{\sim} \frac{n^{2r}}{ \mathrm{vol}(\cJ(V_\R))}\cdot\int_{U}\alpha\wedge \omega_{Betti}^r~.
\end{equation}
\end{theorem}

 \begin{proof}
Let $U\subset \mathcal{M}$ be a simply connected open subset as in the theorem and let $\nu_\R:U\rightarrow V_\R$ be the restriction of the Betti map to $U$. By \Cref{generic-rank-betti}, the generic rank of $\nu_\R$ is equal to $\dim_\R V$, hence $\nu_\R$ is submersive, outside a measure zero subset. Notice also that the preimages of $V_\Q$ are exactly the components in $U$ of the torsion locus of $ \nu$, which shows the analytic density of $\mathrm{TL}(\nu)$.

As for the equidistribution, one can argue exactly as in the proof of \cite[Theorem 3.6]{tayoutholozan} using the submersivity of $\nu_\R$, to prove that $\nu_\R^{-1}(V_\Q)$ is equidistributed with respect to $\nu_\R^{*}\beta_{Betti}^r=\omega_{Betti}^r$.  That is, \eqref{eq-equi} holds
for every differential form $\alpha$, which is the desired result. 
\end{proof}

\begin{remarque}\label{rem-density}
The conclusions of Theorem \ref{density} hold as long as $\dim(\cM)\geq \mathrm{rk}(\cJ(\cV))$ and $\mathrm{rk}(\cJ(\cV))=k$ (i.e.,~$\mathrm{rk}(d\nu_{\RR})=\dim_{\QQ}(V)$).  If $\cV$ is isotrivial, one may be able to check the latter condition directly; for example, it obviously holds if $\mathrm{rk}(\cJ(\cV))=1$ and $\nu$ is not flat.
\end{remarque}

\section{Applications}
\subsection{The torsion locus for the Ceresa cycle}\label{S4.1}
For a curve $C$ of genus $g$ and $s\in C$ a complex base point, we have the Ceresa cycle $Z_C$ and associated normal function $\nu$ as described in the introduction.  In the following we work with a level $\ell\geq 3$ structure (in the automorphic sense), so that the period map is 2:1 and stacky issues do not enter.  By the ``Main Theorem'' below we shall mean Theorem \ref{main}, though the reader can also take it to mean Theorem \ref{main-intro} (which also applies in each instance).

\begin{corollary} Let $\phi\colon \mathcal{M}_g[\ell]\to \mathcal{A}_g[\ell]=X$ be the usual period map, $\mathcal{V}$ be the VHS whose fiber over $[C]\in \mathcal{M}_g[\ell]$ is $H^3_{\text{pr}}(J(C))(2)\cong \textstyle(\bigwedge^3 H^1(C))(2)/(H^1(C)(1))$, and $\nu\in \mathrm{ANF}_{\cM_g[\ell]}(\cV)$ be the Ceresa normal function.  Write $H_g\subset \mathcal{M}_g[\ell]$ for the hyperelliptic locus \textup{(}along which $\Phi$ is ramified\textup{)}.  Then there is a finite set of Shimura subvarieties $Y_i\subset X$ of codimension $\geq 2$, and a finite set of subvarieties $T_i\subset \mathcal{M}_g[\ell]$ of dimension $\leq 2g-2$ \textup{(}or $3$, if $g=3$\textup{)}, such that if $U\subset \mathcal{M}_g[\ell]$ denotes the complement of $H_g$, $\cup T_i$, and $\cup \phi^{-1}(Y_i)$, then $\mathrm{TL}(\nu)\cap U(\CC)$ is a countable set of points.

If one is interested in the locus where the \emph{cycle itself} is torsion, then in the above one can remove any points not in $U(\bar{\QQ})$.  \textup{(}One can make sense of this over $\mathcal{M}_g[\ell]$ by choosing the canonical class as ``base point''.\textup{)}
\end{corollary}

\begin{proof}
The first point is that $\cF^{-1}\cV\neq \cV$ because $\cV$ is irreducible of (Hodge-theoretic) level three.  The remaining assumptions of the Main Theorem are satisfied by the Torelli theorem and Ceresa's theorem (that the Ceresa normal function is nontorsion).  The (co)dimension bounds are from a result of \cite{CP}.  Finally, if a non-$\bar{\QQ}$-point is contained in the torsion locus of the family of cycles themselves, then so is its $\bar{\QQ}$-spread; this does not apply to the normal function.
\end{proof}

\begin{remarque}
The rank of the intermediate Jacobian bundle $\cJ(\mathcal{V})$ is $\tfrac{1}{2}(\binom{2g}{3}-2g)$, which exceeds $3g-3$ for every $g\geq 3$.  So in fact both (i) and (ii) hold in Theorem \ref{main}.  While we don't need (ii), the strict inequality means that zero-dimensional components of the torsion locus are atypical so that one \emph{expects} only finitely many.
\end{remarque}

\begin{example} We know from \cite{laga-schnidman-vanishing} that in genus 3, there is a ball quotient $Y\subset \mathcal{A}_3$ ruled by rational curves, countably many of which are contained in $\Phi(\mathrm{TL}_{\text{pos}}(\nu))$ for the Ceresa normal function. The torsion orders here are unbounded and the torsion locus is Zariski dense, as predicted by \Cref{density} and Remark \ref{rem-density} (as $\nu|_Y$ takes values in the Jacobian of an isotrivial VHS of rank 2); while the fact that this dense union is contained in (the preimage of) a Shimura subvariety of $\cA_3$ illustrates the main theorem. 
\end{example}


\subsection{Normal functions over locally symmetric spaces}\label{S4.2}

In general there are no nontrivial admissible normal functions on (most) locally symmetric varieties as a consequence of Raghunathan's theorem, cf.~\cite{KK}.  Instead, one passes to a finite branched cover (or, deleting the branch locus, an \'etale neighborhood) and considers normal functions there with values in the basechange of a Hermitian homogeneous VHS $\cV=\widetilde{\cV}^{\lambda}\{\tfrac{a}{2}\}$ over $X=\Gamma\backslash D=\Gamma\backslash G_{\RR}/K$.  Such variations arise from an irreducible finite-dimensional representation of $G_{\RR}$ of highest weight $\lambda$ (which we assume has descent to $\ZZ$); for $G$ of Cartan type $A_n$, $D_{\text{odd}}$, or $E_6$, the generalized half-twist $\{\tfrac{a}{2}\}$ is needed to make the Hodge numbers $(p,q)$ integral.  See [op.~cit.] for details.  Of course, we can always apply a Tate twist to put $\cV$ into negative weight.

\begin{corollary}\label{cor-LSV}
Let $\phi\colon \cM\to X$ be an \'etale neighborhood of an irreducible locally symmetric variety of dimension $>1$.  The basechange by $\phi$ of an irreducible Hermitian homogeneous $\ZZ$-VHS $\cV$ \textup{(}as above\textup{)} tautologically has $\phi$ as its period map.  Assume that $\cV$ has negative weight and \textup{(}Hodge-theoretic\textup{)} level $>1$, and consider a nontorsion admissible normal function $\nu\in\mathrm{ANF}_{\mathcal{M}}(\mathcal{V})$.

Then $\mathcal{T}:=\overline{\mathrm{TL}_{\text{pos}}(\nu)}^{\text{Zar}}\subset \mathcal{M}$ is a proper algebraic subvariety.  Moreover, any component of $\mathcal{T}$ is either contained in the preimage of a proper locally symmetric subvariety of $X$, or is a component of $\mathrm{TL}_{\text{pos}}(\nu)$.
\end{corollary}

\begin{proof}
Obviously $\cV$ is non-isotrivial and $\phi$ is quasi-finite, and level $>1$ forces $\cF^{-1}\cV\neq \cV$.  So this is an immediate consequence of the Main Theorem.
\end{proof}

\begin{remarque}
The $\tilde{\cV}^{\lambda}\{\tfrac{a}{2}\}$ admitting infinitesimal normal functions (a necessary condition for the existence of nontorsion $\nu$) have been classified in \cite{KK,CKL}.  For $\dim(X)>1$, only weights $-1$ and $-2$ occur, i.e.,~$r=0$ (``usual'' $K_0$-cycles) or $1$ ($K_1$-cycles) in the sense of Remark \ref{rem-NF}.  At least for $D$ a tube domain, one can do a case-by-case check that when $\cV$ has level $>1$, we always have $\mathrm{rk}(\cJ(\cV))>\dim(D)$, so that both (i) and (ii) hold in Theorem \ref{main}.
\end{remarque}

\begin{example} 
(i) The Ceresa normal function over $\cM_3[\ell]$ is the obvious example, but the point of Corollary \ref{cor-LSV} is to go beyond Ceresa.  So here are two more:

(ii) The Prym-Ceresa normal function over Faber's moduli space $\mathcal{R}^3_4$ of 3:1 \'etale covers $\tilde{C}\overset{\Phi}{\to}C$ of genus $4$ curves.  Here the period map is to the moduli space of Weil abelian 6-folds, given by $A_{\Phi}:=J(\tilde{C})/J(C)$; and the cycle is the projection of Ceresa for $\tilde{C}$ to $A$.  See \cite[\S3.5]{KK} for details.

(iii) The Collino normal function \cite{Collino} over an \'etale neighborhood of $\cM_2$ (namely, the 15-to-1 cover of the locus where $C\in \cM_2$ has only one involution, parametrized by triples $(C,p,q)$ with $p,q\in C$ distinct Weierstrass points.)  Here $\mathcal{V}$ has weight $-2$ and corresponds to a $K_1$/`higher' cycle in the hyperelliptic Jacobian.
\end{example}

\subsection{The Fano normal function}\label{S4.3}

Similar in spirit to the Ceresa normal function is the \emph{Fano normal function} $\nu$ over the (10-dimensional) moduli space $\cM$ of cubic threefolds.  Given a threefold $\bT$ and fixing a line $L_0\subset \bT$, the Abel-Jacobi images $\mathrm{AJ}_{\bT}(L-L_0)\in \cJ(H^3(\bT)(2))=:J(\bT)$ (over all lines $L\subset \bT$) sweep out a surface $F\subset \bT$.  Setting $F^-:=\imath(F)$, the cycle $Z_{\bT}:=F-F^-$ is homologous to zero, and its Abel-Jacobi class $\mathrm{AJ}_{J(\bT)}(Z_{\bT})\in \cJ(H^5_{pr}(J(\bT))(3))$ (resp.~the HS $H^5_{pr}(J(\bT)(3))$) defines $\nu$ (resp.~$\cV$) at $[\bT]\in \cM$.  Then $\phi\colon \cM\to\cA_5$ is an embedding, $\cV$ is irreducible level 5 (whence $\cF^{-1}\cV\neq \cV$), and it is shown in \cite{CNP} that $\nu\in \mathrm{ANF}_{\cM}(\cV)$ is nontorsion.  So the conclusion of Theorem \ref{main-intro} holds.

Even less is known about the torsion locus for the Fano normal function than for Ceresa.  However, if we extend $\nu$ to the closure $\overline{\cM}$ of $\phi(\cM)$, then $\overline{\cM}$ contains the hyperelliptic Jacobian locus $H_5$ and the extended normal function is torsion (in fact zero) on this locus [op.~cit.].

\subsection{Fields of definition of the torsion locus}
In this section, we would like to determine the field of definition of the ``isolated'' components of the torsion locus of a normal function. We first recall a general result about local systems on algebraic varieties. 

Let $S$ be a smooth complex quasi-projective algebraic variety and let $\bV$ be a complex local system on $S$. For any subvariety $Y\subset S$, we denote by $\cH_Y$ the monodromy group of $Y$.\footnote{Defined as the neutral component of the monodromy representation of ($\pi_1$ of) the normalization of $Y$.}

\begin{definition}Let $Y\subseteq S$ be a closed algebraic subvariety. We say that $Y$ is weakly non-factor for $\bV$, if there is no $Z\supset Y$ for which $\cH_Y$ is a strict normal subgroup of $\cH_Z$.
\end{definition}

\begin{theorem}[{\cite[Theorem 2.6]{klingler-otwinowska-urbanik}}]\label{galois}
Let $S$ be a smooth quasi-projective variety defined over a number field $K\subset \C$ and let $\bL$ be a complex local system on $S(\C)^{an}$, whose associated flat bundle is also defined over $K$. Then any weakly special, weakly non-factor subvariety for $\bL$ is defined over $\overline{\QQ}$ and its Galois conjugates are also weakly non-factor, weakly special subvarieties. 
\end{theorem}

This leads at once to the main result of this subsection.

\begin{theorem}
Suppose we have:
\begin{itemize}[leftmargin=0.5cm]
\item $\mathcal{M}$ a smooth quasi-projective variety defined over a number field $K$;
\item $\cV$ a non-isotrivial $\Q$-simple polarizable variation of Hodge structure of negative weight on $\cM$, with associated flat bundle defined over $K$, $\QQ$-simple derived generic Mumford-Tate group $G$, and associated period map quasi-finite;
\item $\nu\in \mathrm{ANF}_{\cM}(\cV)$ a nontorsion admissible normal function arising from a family of cycles defined over $K$; and 
\item $Y$ a positive-dimensional component of the torsion locus of $\nu$.
\end{itemize}
Then either $Y$ is contained in a component of the Hodge locus for $\cV$ or it is defined over $\overline{\QQ}$. In general, the components of the torsion locus which are weakly non-factor are defined over $\overline{\QQ}$ and their Galois conjugates are again components of the torsion locus.
\end{theorem}

\begin{proof} 
The hypothesis on the normal function implies that the flat bundle associated to the extension $0\to \cV\to \cE_{\nu}\to \QQ\to 0$ is defined over $K$.  Thus the associated local system $\bE_{\nu,\CC}$ satisfies the hypothesis of \Cref{galois}.  By \cite{Andre}, the geometric monodromy group of $\bV$ is normal in $G$; as $\bV$ is non-isotrivial, they are therefore equal.  Since $\nu$ is nontorsion, we conclude that the geometric monodromy group $\cH$ of $\bE_{\nu}$ is an extension of $G$ by $V$ ($=$ a fiber of $\bV$).
 
A component $Y$ of the torsion locus has geometric monodromy group $\cH_Y\leq G$, hence (as $\cH_Y<\cH$) it is weakly special. Furthermore, if $Y$ is not contained in a component of the Hodge locus for $\cV$, its derived generic mixed Mumford-Tate group is isomorphic to $G$.  So by \cite{Andre} we must have $\cH_Y\trianglelefteq G$, and as $G$ is simple and the hypotheses force $\cH_Y$ to be nontrivial, $\cH_Y=G$.  By non-isotriviality and irreducibility of $\bV$ on $\cM$, $G$ is its own normalizer in $\cH$.  Thus $Y$ is weakly non-factor and we conclude by \Cref{galois}.
\end{proof}

\begin{corollary}
The \textup{(}finitely many\textup{)} positive dimensional components of the torsion locus of the Ceresa normal function on $\mathcal{M}_g$ with Mumford-Tate group equal to $\mathrm{GSp}_{2g}$ are defined over $\overline \Q$. 
\end{corollary}


\bibliographystyle{alpha}
\bibliography{bibliographie}
\end{document}